\documentclass[11pt,reqno]{amsart}

\usepackage[T1]{fontenc}

\usepackage{amsmath}								%math
\usepackage{amssymb}
\usepackage{amsthm}
\usepackage{amscd}
\usepackage{amsfonts}
\usepackage{stmaryrd}

\usepackage{euler}

\usepackage{extarrows}

\usepackage[colorlinks, linktocpage, citecolor = blue, linkcolor = blue]{hyperref}
\usepackage{color}
\usepackage{tikz}									
\usetikzlibrary{matrix}
\usetikzlibrary{patterns}
\usetikzlibrary{matrix}
\usetikzlibrary{positioning}
\usetikzlibrary{decorations.pathmorphing}
\usetikzlibrary{shapes}

\usepackage{fullpage}
\usepackage{enumerate}

\newtheorem{theorem}{Theorem}[section]
\newtheorem{lemma}[theorem]{Lemma}
\newtheorem{proposition}[theorem]{Proposition}
\newtheorem{corollary}[theorem]{Corollary} 

\theoremstyle{definition}
\newtheorem{definition}[theorem]{Definition}
\newtheorem{example}[theorem]{Example}

\newtheorem{remark}[theorem]{Remark}

\newcommand{\twoheadlongrightarrow}{\relbar\joinrel\twoheadrightarrow}

\newcommand{\R}{\mathbb{R}}

\newcommand{\Z}{\mathbb{Z}}
\newcommand{\ZZ}{\mathbb{Z}}

\newcommand{\CC}{\mathbb{C}}
\newcommand{\N}{\mathbb{N}}

\newcommand{\NN}{\mathbb{N}}

\newcommand{\Mbar}{\overline{M}}

\newcommand{\calA}{\mathcal{A}}

\newcommand{\calC}{\mathcal{C}}

\newcommand{\calL}{\mathcal{L}}

\newcommand{\calO}{\mathcal{O}}
\newcommand{\calS}{\mathcal{S}}

\DeclareMathOperator{\Spec}{Spec}

\DeclareMathOperator{\Pic}{Pic}
\DeclareMathOperator{\Div}{Div}
\DeclareMathOperator{\mdeg}{mdeg}

\DeclareMathOperator{\res}{res}

\title[Logarithmic Picard groups, chip firing, and the combinatorial rank]{Logarithmic Picard groups, chip firing, and the combinatorial rank}

\author{Tyler Foster}
\address{Max Planck Institute for Mathematics\\
Vivatsgasse 7\\
53111 Bonn\\
Germany}
\email{\href{mailto:foster@mpim-bonn.mpg.de}{foster@mpim-bonn.mpg.de}}

\author{Dhruv Ranganathan}
\address{Department of Mathematics\\
Massachusetts Institute of Technology\\
77 Massachusetts Avenue\\
Cambridge MA 02138}
\email{\href{mailto:dhruvr@mit.edu}{dhruvr@mit.edu}}

\author{Mattia Talpo}
\address{Department of Mathematics\\
Simon Fraser University\\
8888 University Drive\\
Burnaby BC\\
V5A 1S6 Canada\\
and Pacific Institute for the Mathematical Sciences\\ 4176-2207 Main Mall \\ Vancouver BC\\ V6T 1Z4 Canada}
\email{\href{mailto:mtalpo@sfu.ca}{mtalpo@sfu.ca}}

\author{Martin Ulirsch}
\address{Department of Mathematics, University of Michigan, Ann Arbor 48109}
\email{\href{mailto:ulirsch@umich.edu}{ulirsch@umich.edu}}

\subjclass[2010]{14H10; 14T05}
\date{\today}
\thanks{T.F.'s research was supported by Institut des Hautes Études Scientifiques, by Le Laboratoire d'Excellence CARMIN, and by a Postdoctoral Fellowship at Max Planck Institute for Mathematics. M.T. was supported by a Postdoctoral Fellowship at the University of British Columbia. M.U.'s research was supported in part by funds by the SFB/TR 45 'Periods, Moduli Spaces and Arithmetic of Algebraic Varieties' of the DFG (German Research Foundation), the Hausdorff Center for Mathematics at the University of Bonn, and the Fields Institute for Research in Mathematical Sciences, Toronto.} 

\begin{document}

\maketitle

\begin{abstract} 
Illusie has suggested that one should think of the classifying group of $M_X^{gp}$-torsors on a logarithmically smooth curve $X$ over a standard logarithmic point as a logarithmic analogue of the Picard group of $X$. This logarithmic Picard group arises naturally as a quotient of the algebraic Picard group by lifts of the chip firing relations of the associated dual graph. We connect this perspective to Baker and Norine's theory of ranks of divisors on a finite graph, and to Amini and Baker's metrized complexes of curves. Moreover, we propose a definition of a \emph{combinatorial rank} for line bundles on $X$ and prove that an analogue of the Riemann-Roch formula holds for our combinatorial rank. Our proof proceeds by carefully describing the relationship between the logarithmic Picard group on a logarithmic curve and the Picard group of the associated metrized complex. This approach suggests a natural categorical framework for metrized complexes, namely the category of logarithmic curves. 
\end{abstract}

\setcounter{tocdepth}{1}
\tableofcontents

%%%%%%%%%%%%%%%%%%%%%%%%%%%%%%%%%%%%%%%%%%%%%%%%%%%%%%

\section{Introduction}

Let $k$ be an algebraically closed field and $S=(\Spec k, k^\ast\oplus \N)$ the standard logarithmic point over $k$. Let $X\rightarrow S$ be a logarithmically smooth curve over $S$ of genus $g$. Suppose furthermore that the logarithmic structure on $X$ is vertical (i.e. it encodes no marked points) and logarithmically semistable (i.e. the total space of every logarithmic smoothing of $X$ is smooth). In \cite{Illusie_logspaces} Illusie proposes that in this situation the group 
\begin{equation*}
\Pic^{log}(X):=H^1(X_{et},M^{gp}_X)
\end{equation*}
should function as an analogue of the algebraic Picard group from the point of view of logarithmic geometry. The group $\Pic^{log}(X)$ parametrizes $M^{gp}_X$-torsors on $X_{et}$, and can be expressed as a natural quotient of the algebraic Picard group $\Pic(X)$ by canonical lifts of the chip firing relations in the theory of divisors on finite graphs. We refer the reader to \cite{Kajiwara_logJacobian} and \cite[Section 3]{Olsson_K3} for a discussion of the geometric properties of the logarithmic Picard group and to Section \ref{section_logPic} below for its precise connection to chip firing on graphs. 

Let $\calL$ be a line bundle on $X$, representing an $M^{gp}_X$-torsor. Inspired by analogous constructions in tropical geometry (see \cite{BakerNorine_RiemannRoch}, \cite{AminiCaporaso_RiemannRoch}, and \cite{AminiBaker_metrizedcurvecomplexes}) we define in Section \ref{section_positivity&rank} below the \emph{combinatorial rank} $r(\calL)=r_{comb}(\calL)$ of $\calL$. We also define the \emph{degree} of $\calL$ as the sum of the degrees of the pullbacks of $\calL$ to the connected components of the normalization of $X$. Denote by $\omega_X^{log}$ the relative logarithmic cotangent bundle on $X$. Our main result is the following Riemann-Roch theorem. 

\begin{theorem}[Riemann-Roch Theorem for logarithmic curves]\label{thm_logRR}
For a line bundle $\calL$ of degree $d$ on $X$ we have
\begin{equation*}
r(\calL) - r(\omega_X^{log}\otimes\calL^{-1}) = d-g+1 \ .
\end{equation*}
\end{theorem}

Our Theorem~\ref{thm_logRR} is proved by a reduction to the theory of ranks of divisors on metrized complexes of algebraic curves, as developed in \cite{AminiBaker_metrizedcurvecomplexes}. Throughout the paper, we refer to these as {\em metrized curve complexes}. In Section  \ref{section_metrizedcurvecomplexes} and Section \ref{section_proof} below we describe a precise correspondence between logarithmic curves and metrized curve complexes, under which every line bundle $\calL$ on $X$ can be represented by a divisor $D$ on the associated metrized curve complex of the same degree and rank. Therefore Theorem \ref{thm_logRR} is a consequence of the Amini and Baker's Riemann--Roch formula for metrized curve complexes (see \cite[Theorem 3.2]{AminiBaker_metrizedcurvecomplexes}). 

%\textcolor{red}{
\begin{remark}
We refrain from considering the case of curves with marked points here (corresponding to non-vertical log smooth curves), since the treatment in \cite{AminiBaker_metrizedcurvecomplexes} does not include metrized curve complexes with legs. We do not know if the Riemann-Roch formula continues to hold in this more general setting.
\end{remark}
%}

Like its classical and tropical counterparts, the above Riemann-Roch theorem has a number of well-known formal consequences, such as the following analogue of the ``easy half'' of Clifford's Theorem:

We say that a line bundle $\calL$ on $X$ is \emph{combinatorially special} if $\omega_X^{log}\otimes\calL^{-1}$ is \emph{combinatorially effective} (see Section \ref{section_positivity&rank} below for details). 

\begin{corollary}[Clifford's Theorem for logarithmic curves]
Let $\calL$ be a combinatorially special line bundle on $X$. Then we have
\begin{equation*}
r(\calL)\leq \frac{\deg(\calL)}{2}\,.
\end{equation*}
Moreover, we have equality if and only if $X$ admits a line bundle of degree $2$ and combinatorial rank $1$. 
\end{corollary}

\noindent
The second statement above follows from the version of Clifford's theorem proved for metrized curve complexes by Len~\cite{Len16} and our comparison between logarithmic curves and metrized complexes. 

If the components of $X$ are smooth (i.e. if they do not have self-intersections), a notion of the ``combinatorial rank" of a line bundle on $X$ has already been proposed by Amini and Baker in \cite[Section 2.2]{AminiBaker_metrizedcurvecomplexes} without using the formalism provided by logarithmic Picard groups. Our definition of a combinatorial rank is a generalization of their definition that applies to general nodal curves (see Corollary \ref{corollary_combinatorialrank=combinatorialrank} below).

\subsection{Motivation and further discussion} In recent years, there has been rapid development in  the theory of linear series for degenerations of curves outside the compact-type locus in the moduli space of curves. Let $\mathscr X\to \Spec(R)$ be a flat, regular, semistable degeneration of a smooth curve over a discrete valuation ring. If the special fiber of $\mathscr X$ has more than one irreducible component, then any line bundle $L$ on the generic fiber of $\mathscr X$ has infinitely many extensions to the special fiber. The combinatorial aspects of the interactions between these extensions can be captured by an independently studied combinatorial tool on the dual graph $G$ of the special fiber of $\mathscr X$, known as \textit{chip-firing} relations. These chip-firing relations can be interpreted as a ``tropical'' notion of linear equivalence for an entirely combinatorial divisor theory on $G$. In this framework we may then study degenerations that have low geometric complexity, for instance, degenerations to rational curves, but high combinatorial complexity in $G$. This theory was pioneered by Baker and Norine~\cite{BakerNorine_RiemannRoch, Baker_specialization} and has led to numerous applications to the geometry and arithmetic of algebraic curves. We direct the reader to the survey~\cite{BakerJensen} and to the references therein, for an introduction to this theory and its applications.

In order to interpolate between the compact-type case considered by Eisenbud and Harris~\cite{EH86a} and the maximal degenerations considered in the tropical approach, Amini and Baker introduced their notion of metrized curve complexes %{\color{red} [I've dealt with Dan's comment here in the paragraph right after Theorem \ref{thm_logRR}. I don't see the utility of saying anything explicit about H. Masur's use of ``metrized curve complex."\ \textemdash\ Tyler]}. 
This theory consists of hybrid objects that are simultaneously tropical and algebraic in nature. Its ``algebraic part'' is a nodal curve, while its ``tropical part'' is a graph. While these objects provide a flexible framework to study degenerations of linear series, so far there has not existed a natural categorical framework in which to study them. 

One of the primary insights of this paper is that logarithmically smooth curves over the standard logarithmic point form a natural and well-established category that contains the theory of metrized curve complexes. Just like a metrized curve complex, any logarithmic scheme $Y$ has an underlying scheme $\underline Y$, as well as a tropicalization $Y^\mathrm{trop}$, via the results of~\cite{U13}. These associations are functorial, and moreover, logarithmic schemes form a category with a good notions of morphisms and moduli functors. In the context of linear series, we suggest here that the notion of a ``logarithmic Picard group'', i.e. the classifying group for torsors of $M^{gp}_X$, has good interaction with the existing tropical and algebraic theories. For instance, Theorem \ref{thm_logRR} reduces to both the classical Riemann-Roch theorem and Baker and Norine's Riemann-Roch theorem for finite graphs in special cases.

\begin{example}[Classical Riemann-Roch]\label{example_classicalRR}
Suppose that the curve $\underline{X}$ underlying $X$ is already smooth. In this case the dual graph of $X$ is a single vertex and hence the chip-firing relations are trivial. We deduce that $\Pic^{log}(X)=\Pic(X)$, and for every line bundle $\calL$ on $X$ the formula $r_{comb}(\calL)=r_{alg}(\calL)=h^0(X,\calL)-1$ holds. Thus Theorem \ref{thm_logRR} reduces to the classical Riemann-Roch formula
\begin{equation*}
h^0(\calL) - h^0(\omega_X\otimes \calL^{-1})= \deg\calL -g +1 \ .
\end{equation*}
\end{example}

\begin{example}[Riemann-Roch for graphs]\label{example_graphRR}
Given a finite graph $G$, one can construct a maximally degenerate logarithmically smooth, semistable, unmarked curve $X_G$ over $S$ such that the dual graph of $X_G$ is $G$.  If the vertices of $G$ have valence either $2$ or $3$, the logarithmic curve $X_G$ is unique up to isomorphism. For any such $X_G$, every divisor $D$ on $G$ arises as the multidegree of a line bundle $\calL$ on $X_G$, the combinatorial rank $r(\calL)$ of $\calL$ is exactly the Baker-Norine rank $r_G(D)$ of $D$ on $G$ (see \cite{BakerNorine_RiemannRoch}), and Theorem \ref{thm_logRR} reduces to the Riemann-Roch formula
\begin{equation*}
r_G(D)-r_G(K_G-D)=\deg D -g +1
\end{equation*}
(see \cite[Theorem 1.12]{BakerNorine_RiemannRoch} or \cite[Theorem 3.6]{AminiCaporaso_RiemannRoch} if $G$ has loops).
\end{example}

In view of Example \ref{example_classicalRR} above it is tempting to suspect that a proof of Theorem \ref{thm_logRR} would provide us with a new proof of the classical Riemann-Roch theorem. However, the proof of the Riemann-Roch theorem for metrized curve complexes in \cite{AminiBaker_metrizedcurvecomplexes} and therefore our proof of Theorem \ref{thm_logRR} employs both the classical Riemann-Roch theorem as well as essential ideas from the proof of Baker-Norine's graph-theoretic Riemann-Roch theorem. 

We close by mentioning two alternate approaches that may help us gain a deeper understanding of the geometry of a metrized curve complex. The first proceeds via Huber's theory of adic spaces, with the tropical connections being established by Foster and Payne~\cite{F15,FP15}. In this approach, one first chooses a semistable formal model $\mathscr X$ of a curve. The collection of all models obtained by admissible formal blowups of $\mathscr X$ forms an inverse system of formal models, and the inverse limit of this system is naturally identified with the Huber adic space $\mathscr X^{\mathrm{ad}}$. By restricting this inverse limit to the subsystem corresponding to modifications that do not change the homeomorphism type of the dual graph, one obtains an \textit{adic skeleton} of $\mathscr X$. The metrized complex associated to $\mathscr X$ can be understood as a partial Hausdorff quotient of this adic skeleton. This gives the metrized complex a topological structure, enriching the originally only set-theoretic definitions. Note that the adic skeleton also comes with a structure sheaf of analytic functions. 

Another approach to understanding metrized complexes proceeds via Parker's theory of \textit{exploded manifolds}. This theory has been developed with a view towards applications in symplectic Gromov--Witten theory. However, every exploded curve in Parker's theory comes with a tropical part and a geometric part. The tropical part of an exploded curve is a dual graph of a prestable curve with a piecewise integer affine structure, while the geometric part is naturally identified with the complex points of a stable curve. In fact, as explained in~\cite[Section 10]{Par12b}, the points of the moduli stack of exploded curves can be naturally identified with the groupoid of logarithmic curves over $(\Spec \ \CC,\CC^*\oplus \NN)$. Under this equivalence, families of exploded curves correspond to families of logarithmic curves over logarithmic bases. 

It is our hope that these results will motivate a further study of the relationship between logarithmic geometry, limit linear series, and tropical Brill--Noether theory, mirroring the development of such a relationship in the theory of stable maps. 

\subsection{Acknowledgements}
The authors would like to thank Jonathan Wise for bringing Illusie's logarithmic Picard group to their attention, sometimes through intermediaries, as well as Dan Abramovich, Yoav Len, and Jeremy Usatine for many discussions related to this topic. This collaboration started at the CMO-BIRS workshop on \emph{Algebraic, Tropical, and Non-Archimedean Analytic Geometry of Moduli Spaces}; many thanks to the organizers Matt Baker, Melody Chan, Dave Jensen, and Sam Payne. Particular thanks are due to Farbod Shokrieh, for his advice to the last author concerning rank-determining sets (as hinted upon in Remark \ref{remark_rankdeterminingsets} below). The document has been improved by the helpful comments of an anonymous referee.

%%%%%%%%%%%%%%%%%%%%%%%%%%%%%%%%%%%%%%%%%%%%%%%%%%%%%%

\section{The logarithmic Picard group and chip firing}\label{section_logPic}

Throughout this article, unless mentioned otherwise, all monoids and logarithmic structures are fine and saturated. In particular, all logarithmic structures are defined on the small \'etale site of a scheme $X$ and all cohomology groups are taken with respect to this \'etale topology. We refer the reader to \cite{Kato_logstr} as well as \cite{Abramovichetal_log&moduli} for the basics of logarithmic geometry. The term \emph{curve} will stand for a one-dimensional reduced and connected scheme that is proper over an algebraically closed field $k$. By a \emph{finite graph} we mean a finite multigraph possibly with loops. % and $1$-valent edges. We refer to $1$-valent edges as {\em half-edges}.
We will think of every edge of the graph as the union of two \emph{half edges}. 
	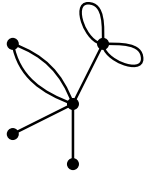
\begin{figure}
	\scalebox{.8}{
	$
	\begin{tikzpicture}
	\draw[black, very thick] (0,0) -- (.5,1);
	\draw[black, very thick] (0,0) to [out=160,in=-80] (-1,1);
	\draw[black, very thick] (0,0) to [out=110,in=-20] (-1,1);
	%\draw[black, very thick] (-1,1) -- (-1.35,1.35);
	%\draw[black, very thick] (-1,1) -- (-1,1.5);
	%\draw[black, very thick] (-1,1) -- (-1.5,1);
	\draw[black, very thick] (0,0) -- (0,-1);
	\draw[black, very thick] (0,0) -- (-1,-.5);
	%\draw[black, very thick] (0,0) -- (.5,0);
	%\draw[black, very thick] (0,0) -- (.35,-.35);
	\draw[black, very thick] (.5,1) to [out=180,in=180] (.25,1.675) to [out=0,in=90] (.5,1);
	\draw[black, very thick] (.5,1) to [out=-90,in=-90] (1.15,.75) to [out=90,in=0] (.5,1);
	\fill[black] (0,0) circle (.1);
	\fill[black] (.5,1) circle (.1);
	\fill[black] (-1,1) circle (.1);
	\fill[black] (0,-1) circle (.1);
	\fill[black] (-1,-.5) circle (.1);
	\end{tikzpicture}
	$
	}
	\caption{Example of a finite multigraph with loops. %and $1$-valent edges
	}
	\end{figure}

Let $P$ be a sharp monoid and denote by $S:=(\Spec k, k^\ast \oplus P)$ the \emph{logarithmic point} with monoid $P$. Let $X\rightarrow S$ be a logarithmically smooth curve over $S$. Recall from \cite[Section 1.1]{Kato_logsmoothcurves} that the underlying curve $\underline{X}$ is nodal and the characteristic monoid $\Mbar_{X,x}=M_{X,x}/\calO_{X,x}^\ast$ at a closed point $x\in X$ is of one of the following three types:
\begin{itemize}
\item[{\bf (i)}] at a node $x_e$, there exists an element $p_{e}\in P$ with associated map $\NN\longrightarrow P$ giving rise to a monoid coproduct $\N^2\oplus_\N P$ (where $\N\to \N^2$ is the diagonal), and we have $\Mbar_{X,x_e}=\N^2\oplus_\N P$; 
\item[{\bf (ii)}] there is a finite number of smooth points $x_1,\ldots, x_n$ on $X$ such that $\Mbar_{X, x_i}=\N\oplus P$, the \emph{marked points} of $X$, and
\item[{\bf (iii)}] for all other points $x$ of $X$ we have $\Mbar_{X,x}=P$. 
\end{itemize}

If we want to emphasize the dependence of the amalgamated sum $\N^2\oplus_\N P$ in (i) on $p_e$ we use the notation $\N^2\oplus_{\Delta,\N,p_e}P$ instead.

\underline{For the remainder of this section} let us assume $P=\N$, that $X$ is \emph{vertical}, i.e. that it has no marked points, and that $X$ is \emph{logarithmically semistable}, i.e. for every node $x_e$ of $X$ we have $p_e=1$ (so that every logarithmic smoothening of $X$ has a smooth total space; see \cite[Section 2]{Olsson_K3}).

We denote by $G=G_X$ the dual graph of $X$. The vertices $V=V(G)$ of $G$ are in one-to-one correspondence with the components $X_v$ of $X$, and we have an edge $e\in E=E(G)$ between two vertices $v,v'$ for each node $x_e$ connecting the two components $X_v$ and $X_{v'}$. Note that we explicitly allow $v=v'$, i.e. a component $X_v$ can intersect itself in a node $x_e$; the corresponding edge is a loop in $G$ emanating from $v$. 
 
Denote by $\pi\mathrel{\mathop:}\widetilde{X}\rightarrow X$ the normalization of $X$, write $X_v$ (respectively $\widetilde{X}_v$) for the irreducible (respectively connected) components of $X$ (respectively $\widetilde{X}$), and let $\pi_v\colon \widetilde{X}_v \to X_v$ denote the restriction of $\pi$.

We have $\Mbar_X^{gp}=\bigoplus_v (\pi_{v})_{\ast}\Z$ and thus $H^1(X, \Mbar_X^{gp})=0$, since $H^1(\widetilde{X}_v, \Z)=0$ for every $v$ (by \cite[IX, Proposition 3.6]{SGA4}). Therefore the fundamental short exact sequence 
\begin{equation*}\begin{CD}
0@>>>\calO_X^\ast @>>>M_X^{gp} @>>>\Mbar_X^{gp} @>>> 0
\end{CD}\end{equation*}
induces a long exact sequence containing the exact sequence
\begin{equation*}\begin{CD}
H^0(X,\Mbar_X^{gp})@>>> H^1(X,\calO_X^\ast) @>>> H^1(X,M^{gp}_X) @>>> 0,
\end{CD}\end{equation*}
and we see that $\Pic^{log}(X)$ is a quotient of $\Pic(X)=H^1(X,\calO_X^\ast)$ by the relations induced by a homomorphism of abelian groups
\begin{equation*}
\Z^V\longrightarrow \Pic(X) \ .
\end{equation*}

\begin{proposition}[\cite{Olsson_K3} Section 3.3] \label{prop_chipfiring}
Denote by $\calL_v$ the image, in $\Pic(X)$, of the basis vector of $\ZZ^{V}$ corresponding to $X_v$. Then there are natural isomorphisms
\begin{equation*}
\calL_v\vert_{X_w} \simeq\calO_{X_w}(D_{vw}) 
\end{equation*}
for $w\neq v$, where $D_{vw}=X_v\cap X_w$, as well as an isomoprhism
\begin{equation*}
\bigotimes_v\calL_v\simeq \calO_X \ .
\end{equation*}
\end{proposition}

Note in particular that Proposition \ref{prop_chipfiring} also determines $\calL_v\vert_{X_v}$, as it induces a natural isomorphism
\begin{equation*}
\calL_v\vert_{X_v}\simeq  \bigotimes_{w\neq v} \calO_{X_v}(-D_{vw}) 
\end{equation*}
for all components $X_v$ of $X$. 

\begin{proof}[Proof of Proposition \ref{prop_chipfiring}]
Denote by $e_v$ the section of $\Mbar_X^{gp}$ corresponding to $1\in \Z$ on the component $X_v$ of $X$. Then $\calL_v$ is  exactly the $\calO_X^\ast$-torsor of lifts of $e_v$ to $M_X^{gp}$. Let $U$ be an \'etale neighborhood around a node $x_e$ of $X$ connecting $X_v$ and $X_w$, small enough so that we may assume that $X$ is isomorphic to
\begin{equation*}
\Spec k[x,y]/(xy) \ ,
\end{equation*}
where $X_v$ is given by $x=0$ and $X_w$ by $y=0$. Suppose we have a lift $\tilde{e}_v$ of $e_v$ to $M_X^{gp}$ on $U$. Then $\alpha(\tilde{e}_v)\in\calO_X(U)$ is equal to $x$ up to a multiple in $\calO_X^\ast$, and this implies that $\calL_v\vert_{X_w}=\calO_U(x_e)$ on $U$. This local argument implies $\calL_v\vert_{X_w}=\calO_{X_w}(D_{vw})$ globally. 

In order to see $\bigotimes_v\calL_v\simeq \calO_X$ we note that the image of $1\in M_S^{gp}$ is a lift of $(1,1,\ldots, 1)\in \Z^V$ in $\Mbar_X^{gp}$ to a global non-vanishing section of $\bigotimes_v\calL_v$. 
\end{proof}

If $\calL$ is a line bundle on $X$, we denote by $[\calL]$ the associated element of the logarithmic Picard group $\Pic^{log}(X)=H^1(X,M^{gp}_X)$. In what follows we work with line bundles on $X$, which we interpret as representatives of isomorphism classes of $M^{gp}_X$-torsors, via the surjection $\Pic(X)\to \Pic^{log}(X)$. We say that two line bundles $\calL$ and $\calL'$ are \emph{combinatorially equivalent} if $[\calL]=[\calL']$.

Proposition \ref{prop_chipfiring} shows that the relations in $\Pic^{log}(X)$ are really an incarnation, in the world of logarithmic geometry, of chip firing between divisors on the dual graph $G$ of $X$. Indeed, recall from \cite{BakerNorine_RiemannRoch} that a divisor $D$ on a finite graph $G$ is a finite formal sum $\sum a_v v$ over the vertices of $G$. Write $\Div(G)$ for the free abelian group of divisors on $G$. The degree of $D$ is given by $\deg D=\sum a_v$. It is useful to think of $D$ as a configuration of positive and negative chips sitting at the vertices of $G$. A \emph{chip-firing move} at a vertex $v$ transforms a divisor $D$ into another divisor $D'$ that is given by transferring one chip along every edge of $G$ emanating from $v$ to the next vertex. So, after one chip-firing move, we have 
\begin{equation*}
a'_v=a_v-\vert v\vert+2\cdot\#\textrm{ loops at } v
\end{equation*}
where $\vert v\vert$ denotes the valence of the vertex $v$, as well as 
\begin{equation*}
a_w'=a_w+\#\textrm{ edges between }v \textrm{ and } w 
\end{equation*}
for $w\neq v$, and therefore $\deg D=\deg D'$. Denote by $\Pic (G)$ the \emph{Picard group} of $G$, i.e. the group of divisors on $G$ modulo the equivalence relation $\sim$ generated by chip-firing moves.

\begin{corollary}
Denote by $G$ the dual graph of $X$. The multi-degree homomorphism
\begin{equation*}\begin{split}
\mdeg\mathrel{\mathop:} \Pic (X)&\longrightarrow \Div G\\
\calL&\longmapsto \big(\deg\pi^\ast\calL\vert_{\tilde{X}_v}\big)_{v\in V(G)}
\end{split}\end{equation*}
descends to a homomorphism
\begin{equation*}
\tau\mathrel{\mathop:}\Pic^{log}(X)\longrightarrow \Pic (G) \ .
\end{equation*}
\end{corollary}

\begin{proof}
In the notation of Proposition \ref{prop_chipfiring} we have that 
\begin{equation*}
\deg \pi^\ast\calL_v\vert_{\widetilde{X}_v}=-\vert v\vert+2\cdot\#\textrm{ loops at } v
\end{equation*}
as well as
\begin{equation*}
\deg\pi^\ast\calL_v\vert_{\widetilde{X}_w}=\#\textrm{ edges between }v \textrm{ and } w \ .
\end{equation*}
Therefore adding $\mdeg(\calL_v)$ to a divisor on $G$ coincides with the operation of carrying out a chip-firing move at the vertex $v$. Since the kernel of $\Pic(X)\twoheadrightarrow \Pic^{log}(X)$ is generated by the $\calL_v$, this implies the claim. 
\end{proof}

Given a line bundle $\calL$ on $X$, we define its \emph{degree} by
\begin{equation*}
\deg(\calL)=\sum_v\deg_{\tilde{X}_v}\pi^\ast \calL\vert_{\tilde{X}_v} \ .
\end{equation*}

\begin{corollary}[\cite{Olsson_K3} Corollary 3.5]\label{corollary_degree}
The degree map descends to a map 
\begin{equation*}
\deg\mathrel{\mathop:}\Pic^{log}(X)\longrightarrow \Z \ .
\end{equation*}
\end{corollary}

\begin{proof}
The degree of a line bundle $\calL$ on $X$ is equal to the degree of $\mdeg\calL$ on $G$ and the claim follows from the invariance of the degree map on $\Div(G)$ under chip firing. 
\end{proof}

\begin{remark}
The combinatorics of chip firing has also appeared in algebraic geometry without the term being explicitly used, most notably in Caporaso's thesis~\cite{Cap94}. The explicit combinatorics and its relation to the component group of the N\'eron model of the Picard variety (see~\cite{Ray70}) can be found in \cite[Section 3]{CaporasoNeron} using the terminology of so-called ``twisters'' for extensions of the trivial line bundle in a semistable degeneration.
\end{remark}

%%%%%%%%%%%%%%%%%%%%%%%%%%%%%%%%%%%%%%%%%%%%%%%%%%%%%%

\section{Combinatorial positivity and rank}\label{section_positivity&rank}

Let $G$ be a finite graph. A divisor $D=\sum_va_vv$ on $G$ is \emph{effective} (written as $D\geq 0$), if $a_v\geq 0$ for all vertices $v$ in $G$. The \emph{linear system} $\vert D\vert$ of $D$ is defined to be the set
\begin{equation*}
\vert D\vert =\big\{D'\geq 0\big\vert D'\sim D\big\} \ .
\end{equation*} 
Recall from \cite[Section 2]{BakerNorine_RiemannRoch} that the \emph{Baker-Norine rank} $r_G(D)$ of a divisor $D$ on $G$ is defined as the biggest integer $r\geq 0$ such that for all effective divisors $D'$ of degree $r$ the linear system $\vert D-D'\vert\neq \emptyset$. If this condition is not fulfilled for any $r\geq 0$, we set $r_G(D)=-1$. %We now use the ideas in \cite{BakerNorine_RiemannRoch} and \cite{AminiBaker_metrizedcurvecomplexes} as motivation in order to define the rank of a logarithmic line bundle. 

We now give an analogous collection of definitions for line bundles on a log curve $X$ as in the previous section (using the ideas in \cite{BakerNorine_RiemannRoch} and \cite{AminiBaker_metrizedcurvecomplexes} as inspiration). A line bundle $\calL$ on $X$ is said to be \emph{combinatorially effective} (written as $\calL\geq 0$), if the restriction $\pi^\ast\calL\vert_{\widetilde{X}_v}$ has non-negative rank for all $v$, i.e. if we have 
\begin{equation*}
r_{alg}\big(\pi^\ast\calL\vert_{\widetilde{X}_v}\big)=h^0\big(X_v,\pi^\ast\calL\vert_{\widetilde{X}_v}\big)-1\geq 0
\end{equation*}
for all $v$. This is equivalent to $\pi^\ast\calL\vert_{\widetilde{X}_v}=\calO_{\widetilde{X}_v}(D)$ for an effective divisor $D$ on $\widetilde{X}_v$.

\begin{definition}\label{definition_rank}
Let $\calL$ be a line bundle on $X$. Define the \emph{combinatorial linear system} $\vert \calL \vert$ of $\calL$ as the set
\begin{equation*}
\vert \calL\vert = \big\{\calL'\geq 0\big\vert [\calL']=[\calL] \big\} \ .
\end{equation*}
The \emph{combinatorial rank $r(\calL)=r_{comb}(\calL)$ of $\calL$} is defined as the greatest integer $r\geq 0$ such that $\vert\calL\otimes(\calL')^{-1}\vert\neq\emptyset$ for all line bundles $\calL'$ of degree $r$. If this condition is not fulfilled for any $r\geq 0$, we set $r(\calL)=-1$. 
\end{definition}

\begin{remark}
The terminology here is meant to parallel the terminology in \cite[Section 2.2]{AminiBaker_metrizedcurvecomplexes}. In fact, we will prove later (Corollary \ref{corollary_combinatorialrank=combinatorialrank}) that when both notions apply, they indeed agree with each other.
\end{remark}

It is immediate that, if two line bundles $\calL$ and $\calL'$ are combinatorially equivalent on $X$, then $\vert\calL\vert=\vert \calL'\vert$ as well as $r(\calL)=r(\calL')$, so that these depend only on the associated $M^{gp}_X$-torsor. 

\begin{proposition}\label{prop_specializationlemma}
Let $X$ be a logarithmic curve that is vertical and logarithmically semistable, and denote by $G$ its dual graph. Given a line bundle $\calL$ on $X$, we have
\begin{equation*}
r(\calL)\leq r_G\big(\tau(\calL)\big)  \ .
\end{equation*} 
\end{proposition}

\noindent Propositions \ref{prop_maxdeg} and \ref{prop_Pic=Pic} below show that, if $X$ is maximally degenerate, the above inequality becomes an equality, i.e. we have
 \begin{equation*}
r(\calL)=r_G\big(\tau(\calL)\big) \ .
\end{equation*}
This, in particular, explains Example \ref{example_graphRR} in the introduction.

\begin{proof}[Proof of Proposition \ref{prop_specializationlemma}]
If $r(\calL)=-1$ there is nothing to show. So we may assume $r(\calL)\geq 0$. We are going to show that, if for $r\geq 0$ we have $r(\calL)\geq r$, then $r_G(D)\geq r$ for $D=\mdeg\calL\in \Div(G)$. Given a divisor $D'$ on $G$ of degree $r$, there is a line bundle $\calL'$ of degree $r$ on $X$ such that $\mdeg(\calL')=D'$, since $\mdeg\mathrel{\mathop:}\Pic^r(X)\rightarrow\Pic^r(G)$ is surjective (as easily checked - see also Proposition \ref{prop_Pic=Pic} below). Since $r(\calL)\geq r$, we have that $\vert \calL\otimes (\calL')^{-1}\vert\neq \emptyset$, i.e. there exists a combinatorially effective line bundle $\calL''$ on $X$ that is equivalent to $\calL\otimes (\calL')^{-1}$. Its image $D''=\mdeg(\calL'')$ is an effective divisor on $G$ that is equivalent to $D-D'$, so $\vert D-D'\vert\neq\emptyset$. This implies $r(D)\geq r$.
\end{proof}

Finally, in order to explain Example \ref{example_classicalRR} from the introduction, suppose that the underlying curve $\underline{X}$ of $X$ is already smooth. Then $X$ has only one component and Proposition \ref{prop_chipfiring} shows that the map 
\begin{equation*}
\Z^V=\Z\twoheadlongrightarrow \Pic(X)
\end{equation*}
is zero. Therefore the quotient above induces an isomorphism
\begin{equation*}
\Pic(X)\simeq\Pic^{\log}(X) \ .
\end{equation*}

\begin{proposition} Assume that $X$ is smooth, and let $\calL$ be a line bundle on $X$. Then we have 
\begin{equation*}
r_{comb}(\calL) = r_{alg}(\calL)=h^0 (X,\calL) - 1 \ .
\end{equation*}
\end{proposition}

\begin{proof}
Choose a divisor $D$ on $X$ such that $\calO(D)=\calL$. We have that $\vert\calL\otimes (\calL')^{-1}\vert\neq\emptyset$ for all line bundles $\calL'$ of degree $r$ on $X$ if and only if $\vert D-D'\vert\neq \emptyset$ for all divisors $D'$ of degree $r$ on $X$, i.e. if and only if $r(D)\geq r$. The claim then follows from \cite[Lemma 2.4]{Baker_specialization}. 
\end{proof}

%%%%%%%%%%%%%%%%%%%%%%%%%%%%%%%%%%%%%%%%%%%%%%%%%%%%%%

\section{Metrized curve complexes and logarithmic curves}\label{section_metrizedcurvecomplexes}

We begin by recalling the definition of metrized curve complexes, as originally introduced in \cite{AminiBaker_metrizedcurvecomplexes}. In fact, inspired by the upcoming \cite{CavalieriChanUlirschWise_tropstack}, we provide a slight generalization of the definition in \cite{AminiBaker_metrizedcurvecomplexes}, where we allow the edge lengths to take values in arbitrary sharp monoids. 

\begin{definition}
Let $P$ be a sharp monoid (not necessarily fine and saturated). A metrized curve complex $\calC$ with edge lengths in $P$ over the field $k$ is given by the following data:
\begin{itemize}
\item a connected finite graph $G$ together with a function $d\mathrel{\mathop:}E(G)\rightarrow P$ that we call {\em edge length};
\item for each vertex $v$ of $G$, a pair $(C_{v},x_v)$ consisting of a complete nonsingular curve $C_v$ and a bijection $f\mapsto x_v^f$ from the half-edges $f$ emanating from $v$ to a finite subset $\calA_v=\{x_v^f\}$ of distinct closed points on $X_v$. 
\end{itemize}
\end{definition}

If $P$ is a submonoid of $\R_{\geq 0}$, this is precisely the notion defined in \cite{AminiBaker_metrizedcurvecomplexes}; in this case it is often useful to think of $\calC$ in terms of its \emph{geometric realization} $\vert \calC\vert$ given as the union of all $C_v$ and all edges $e$ (of length $d(e)$) where the endpoint of an edge $e$ at a vertex $v$ is identified with $x_v^e\in C_v$. 

We may associate to every metrized curve complex $\calC$ with values in $P$ its underlying weighted metric graph $\Gamma$ by considering the vertex-weight 
\begin{equation*}\begin{split}
h\mathrel{\mathop:}V(G)&\longrightarrow \Z_{\geq 0}\\
v&\longmapsto g(C_v) \ .
\end{split}\end{equation*}
The \emph{genus} of $\calC$ is defined to be the number 
\begin{equation*}
g(\calC)=b_1(G)+\sum_{v\in V}h(v) \, ,
\end{equation*}
where $b_1(G)=\#E(G)-\#V(G)+1$ is the Betti number of the graph $G$.

\begin{remark}
In addition to submonoids of the real numbers, another natural class of monoids to work with are \textit{valuated} monoids, for instance, the non-negative elements in the abelian group $\mathbb{R}^k$ for the lexicographic order. The metrized complexes thus obtained arise naturally from skeletons of Hahn analytic spaces, as introduced and studied in~\cite{FR15}. 
\end{remark}

\begin{definition}
An \emph{automorphism} $\phi$ of a metrized curve complex $\calC$ consists of the following data:
\begin{itemize}
\item an automorphism $\phi_\Gamma$ of the underlying graph $G$;
\item a collection of isomorphisms $\phi_v\mathrel{\mathop:}C_v\xrightarrow{\sim} C_{\phi_\Gamma(v)}$,
\end{itemize}
satsfying
\begin{itemize}
\item $d(e)=d\big(\phi_\Gamma(e)\big)$ for all edges $e$ of $G$;
\item $x^{\phi_{\Gamma}(f)}_{\phi_{\Gamma}(v)}=\phi_{v}(x^{f}_{v})$ for all half-edges $f$ at any vertex $v$ of $G$.
\end{itemize}
\end{definition}

A \emph{marking of $\calC$} is a choice of $n$ points $p_1,\ldots, p_n$ on $\bigsqcup_vC_v-\calA_v$. An automorphism $\phi$ of $(\calC,p_1,\ldots, p_n)$ is an automorphism $\phi$ of $\calC$ such that $\phi(p_i)=p_i$.  

Denote by $\calS_{g,n}$ the category fibered in groupoids of logarithmically smooth curves of genus $g$ with $n$ marked points, as introduced in \cite{Olsson_logtwistedcurves}. By \cite[Lemma 5.1]{Olsson_logtwistedcurves} $\calS_{g,n}$ is an algebraic stack. It contains an open substack $\calS_{g,n}^\circ$ parametrizing smooth $n$-marked curves and the complement of $\calS_{g,n}^{\circ}$ in $\calS_{g,n}$ has normal crossings, making $\calS_{g,n}$ into a logarithmically smooth Artin stack.

\begin{theorem}\label{thm_logcurve=mcc}
Let $P$ be a fine and saturated monoid that is sharp. Denote by $S=(\Spec k, k^\ast\oplus P)$ the logarithmic point with monoid $P$. There is a natural equivalence of groupoids  
\begin{equation*}
\calS_{g,n}(S)\xlongrightarrow{\sim}\left\{\begin{array}{c}
   \textrm{metrized curve complexes of genus } g \\ \textrm{with } n
    \textrm{ marked points } \textrm{and edge lengths in } P\\
  \end{array} \right\}. 
\end{equation*}
\end{theorem}

\begin{proof}
To a logarithmically smooth curve $X$ over $S$ we associate a metrized curve complex $\calC(X)$ with edge lengths as follows:
\begin{itemize}
\item Its underlying graph $G_X$ is the dual graph of $X$. A vertex $v$ of $G_X$ corresponds to an irreducible component $X_v$ of $X$, and $G_X$ contains an edge $e$ between two vertices $v$ and $w$ for every node $x_e$ that is contained in the components $X_v$ and $X_w$.
\item The length of an edge $e$ is the element $p_e\in P$ given by $1\mapsto p_e$ in the morphism $\N\rightarrow P$ determining the second summand in
\begin{equation*}
\Mbar_{X,x_e}=\N^2\oplus_\N P \ . 
\end{equation*}
\item To a vertex $v$ in $G_X$ we associate the normalization $\widetilde{X}_v$ of the corresponding irreducible component. If  $e$ is a not a loop we associate to $e$ the unique preimage of the node $x_e$ in $\widetilde{X}_v$. Otherwise, if $e$ is a loop, it consists of two half-edges $f_1$ and $f_2$ emanating from $v$, to which we associate  the two preimages $x_{f_i,v}$ (for $i=1,2$) of $x_e$ in $\widetilde{C}_v$. 
\item Finally, each marked point $x_{l_i}$ in $X$ becomes a marked point on the normalization $\widetilde{X}_v$ of the component $X_v$ that $x_{l_i}$ is contained in.
\end{itemize}

Conversely, we may associate to a metrized curve complex $\calC$ with edge lengths in $P$ a logarithmically smooth curve $X_\calC$ as explained below:
\begin{itemize}
\item The underlying algebraic curve $\underline{X_\calC}$ is a nodal configuration given by glueing every $C_v$ to $C_w$ along a node at the points $x_{f,v}$ and $x_{f',w}$, whenever the half edges $f$ and $f'$ form an edge in $G_\calC$. 
\item The curve $\underline{X_\calC}$ carries a minimal logarithmic structure $M'$ over the standard logarithmic point $S'=\big(\Spec k,k^\ast\oplus \N^{ E(G_\calC)}\big)$, which also encodes the marked points $x_1,\ldots, x_n$. The monoid homomorphism
\begin{equation*}\begin{split}
\N^{ E(G_\calC)}&\longrightarrow P \\
\vec{e}&\longmapsto p_e
\end{split}\end{equation*}
induces a map $S\rightarrow S'$, and we define the logarithmic structure $M_X$ on $X_\calC$ as the pullback of $M_X'$ along this map.  
\end{itemize}

The two associations $X\mapsto \calC(X)$ and $\calC\mapsto X_\calC$ are mutually inverse (up to isomorphism). An automorphism $\phi$ of a logarithmic curve $X$ over $S$ naturally induces an automorphism $\phi_G$ of the weighted dual graph as well as isomorphisms $\widetilde{X}_v\xrightarrow{\sim} \widetilde{X}_{f(v)}$ for every vertex of $G$. Moreover, the automorphism $f$ induces an isomorphism 
\begin{equation*}
\N^2\oplus_{\Delta,\N,p_e}P= \Mbar_{X,x_e}\xrightarrow{\sim}\Mbar_{X,\phi(x_e)}=\N^2\oplus_{\Delta,\N,p_{\phi_G(e)}}P
\end{equation*}
for every node $x_e$ of $X$ and this implies that $d\big(\phi_G(e)\big)=p_e=d(e)$ for every edge $e$ of $G$. This shows that $\phi$ induces a natural automorphism $\phi_{\calC(X)}$ of $\calC(X)$. If $X$ has marked points $\phi$ preserves their marking and so does the induced automorphism $\phi_{\calC(X)}$. The association $\phi\mapsto \phi_{\calC(X)}$ is clearly functorial in $\phi$ and every automorphism of $\calC(X)$ uniquely determines an automorphism of $X$. This finishes the proof of Theorem \ref{thm_logcurve=mcc}.
\end{proof}

%%%%%%%%%%%%%%%%%%%%%%%%%%%%%%%%%%%%%%%%%%%%%%%%%%%%%%

\section{Comparing Picard groups}\label{section_proof}
Let us recall from \cite{AminiBaker_metrizedcurvecomplexes} the theory of ranks of divisors on metrized curve complexes. For simplicity we restrict our attention from now on to a metrized curve complex $\calC$ with edge lengths in $\N$, all equal to one. 

A \emph{divisor} $D$ on $\calC$ is a tuple $\big(D_G, (D_v)_{v\in V(G)}\big)$ consisting of a divisor $D_G=\sum_v a_v v$ on $G$ as well as divisors $D_v$ on every component $C_v$ of $\calC$, such that $\deg D_v=a_v$ for all vertices $v$ of $G$. Note that because of the last equality, $D_G$ is actually determined by the family of divisors $(D_v)_{v\in V(G)}$. Denote by $\Div(\calC)$ the abelian group of divisors on $\calC$ and by $\deg D=\deg D_G$ the \emph{degree} of $D$. Generalizing the chip firing operations on $G$, there is an equivalence relation on $\Div(\calC)$ generated by the following two operations:
\begin{enumerate}
\item For every vertex $v$ of $G$, we may replace $D_v$ by a divisor $D_v'$ that is equivalent to $D_v$ on the smooth curve $C_v$.
\item Given a vertex $v$, we may apply a chip firing move to transform $D_G$ into $D_G'$ (see the discussion in Section \ref{section_logPic}), in which case we obtain
\begin{equation*}
D_v'=D_v-\sum_{\textrm{half edges } f \textrm{ at } v}  x_{f,v},
\end{equation*}
as well as 
\begin{equation*}
D_w'=D_w+\sum_{\textrm{edges } e \textrm{ between } v \textrm{ and } w} x_{e,v}
\end{equation*}
for consistency.
\end{enumerate}

The quotient of $\Div(\calC)$ by the above relations is called the \emph{Picard group} of $\calC$, and will be denoted by $\Pic(\calC)$. A divisor $D$ on $\calC$ is said to be \emph{effective} (written as $D\geq 0$) if both $D_G$ and all $D_v$ are effective. Denote by 
\begin{equation*}
\vert D\vert =\big\{D'\geq 0\big\vert D'\sim D\big\}
\end{equation*}
the linear system associated to $D$. The \emph{rank} $r_\calC(D)$ of $D$ is the greatest integer $r\geq 0$ such that $\vert D-D'\vert\neq\emptyset$ for all divisors $D'$ on $\calC$ of degree $r$. If this condition is not fulfilled for any $r\geq 0$, we set $r_\calC(D)=-1$. 

It is well-known (see e.g. \cite{AminiBaker_metrizedcurvecomplexes}), that the inequality $r_\calC(D)\leq r_G(D_G)$ holds. The following Proposition \ref{prop_maxdeg} is a partial converse to this inequality; it is used in the proof of Proposition \ref{prop_specializationlemma} above. 

\begin{proposition}\label{prop_maxdeg}
If $\calC$ is maximally degenerate, i.e. if all components $C_v$ are projective lines, we have 
\begin{equation*}
r_\calC(D)=r_G(D_G) 
\end{equation*}
for a divisor $D=\big(D_G,(D_v)_{v\in V(G)}\big)$ on $\calC$. 
\end{proposition}

\begin{proof}
We are going to show that $r_G(D_G)\geq r$ implies $r_\calC(D)\geq r$. The case $r=-1$ is trivial, so we may assume $r\geq 0$. Let $D'$ be a divisor on $\calC$ of degree $r$. Since $\vert D_G-D_G'\vert\neq\emptyset$, there is an effective divisor $D_G''=\sum_va_v'' v$ on $G$ that is equivalent to $D_G-D_G'$. We may lift $D_G''$ to a divisor $D$ on $\calC$ by choosing any effective $D_v''$ in the linear system $\vert \calO_{C_v}(a_v'')\vert$ on $C_v$. As $\Pic(C_v)=\Z$, this divisor has to be equivalent to $D_v-D_v'$ on $C_v$, and so we have $\vert D-D'\vert\neq\emptyset$. This implies $r_\calC(D)\geq r$.   
\end{proof}

\begin{remark}\label{remark_rankdeterminingsets}
Amini and Baker develop their theory in \cite{AminiBaker_metrizedcurvecomplexes} for arbitrary real edge lengths. As explained in \cite[Section 2.1]{AminiBaker_metrizedcurvecomplexes} the resulting rank is equivalent to one just presented. This can be seen using the theory of \emph{rank determining sets} for metrized complexes, which can be found in \cite[Appendix A]{AminiBaker_metrizedcurvecomplexes}, and is based on earlier ideas in \cite{HladkyKralNorine_ranksofdivisors} and \cite{Luo_rankdeterminingsets} for tropical curves and metric graphs.
\end{remark}

\begin{proposition}\label{prop_Pic=Pic}
Let $X$ be a logarithmically smooth curve over $S=(\Spec k,k^\ast\oplus \N)$ that is logarithmically semistable and vertical. Denote by $\calC=\calC(X)$ the corresponding metrized curve complex. 
\begin{enumerate}[(i)]
\item There is natural short exact sequence
\begin{equation*}\begin{CD}
0@>>> (k^\ast)^{b_1(G)}@>>>\Pic^{\log}(X)@>>> \Pic(\calC)@>>> 0 \ ,
\end{CD}\end{equation*} 
where $b_1(G)=\#E-\#V+1$ is the Betti number of the dual graph $G=G_X$ of $X$, i.e. $\Pic^{log}(X)$ is an extension of $\Pic(\calC)$ by the split algebraic torus $(k^\ast)^{b_1(G)}$. 
\item The epimorphism $\Pic^{\log}(X)\twoheadrightarrow \Pic(\calC)$ preserves degrees and ranks, i.e. given a line bundle $\calL$ on $X$, whose class $[\calL]$ in $\Pic^{log}(X)$ maps to the equivalence class $[D]$ of a divisor $D$ on $\calC_X$, we have 
\begin{equation*}
\deg (\calL)=\deg(D), \qquad \textrm{as well as}\qquad r_{comb}(\calL)=r_{\calC}(D) \ .
\end{equation*} 
\end{enumerate}
\end{proposition}

\begin{proof}
The short exact sequence 
\begin{equation*}\begin{CD}
0@>>>\calO_X^\ast@>>>\pi_\ast \calO_{\widetilde{X}}^\ast @>>> \bigoplus_{e\in E(G)}k^\ast_{x_e}@>>> 0
\end{CD}\end{equation*}
induces a long exact sequence
\begin{equation*}\begin{CD}
0@>>>k^\ast @>>> (k^\ast)^{V(G)} @>>> (k^\ast)^{E(G)} @>>> H^1({X},\calO_X^\ast)@>>> H^1({X},\pi_\ast \calO_{\widetilde{X}}^\ast) @>>> 0 \ ,
\end{CD}\end{equation*}
which in turn gives rise to a natural short exact sequence
\begin{equation*}\begin{CD}
0@>>> (k^\ast)^{b_1(G)}@>>> \Pic(X) @>>> \Pic(\widetilde{X})@>>> 0.
\end{CD}\end{equation*}
This shows that $\Pic(X)$ is an extension of $\Pic(\widetilde{X})$ by the split algebraic torus $(k^\ast)^{b_1(G)}$. 

Note that $\Div(\widetilde{X})=\Div(\calC)$. By Propositon \ref{prop_chipfiring}, the Picard group $\Pic(\calC)$ of $\calC$ naturally arises as the quotient of $\Pic(\widetilde{X})$ by the chip firing relations that are induced by the composition 
\begin{equation*}\begin{CD}
\Z^V@>>> \Pic(X) @>>>\Pic(\widetilde{X}) \ .
\end{CD}\end{equation*}
Thus we have a natural short exact sequence
\begin{equation*}\begin{CD}
0@>>>\Z^{\#V-1}@>>>\Pic(\widetilde{X})@>>>\Pic(\calC)@>>> 0,
\end{CD}\end{equation*}
which fits into a commutative diagram
\begin{equation*}\begin{CD}
@. 0 @. 0 @. 0 @.\\
@. @| @VVV @VVV @.\\
0@= 0@>>> \Z^{\#V-1} @=\Z^{\#V-1} @>>> 0\\
@. @VVV @VVV @VVV @.\\
0@>>> (k^\ast)^{b_1(G)}@>>> \Pic(X) @>>>\Pic(\widetilde{X})@>>> 0\\
@. @| @VVV @VVV @.\\
0@>>> (k^\ast)^{b_1(G)} @>>> \Pic^{log}(X) @>>> \Pic(\calC) @>>> 0\\
@. @VVV @VVV @VVV @.\\
 @. 0 @. 0 @. 0. @. 
\end{CD}\end{equation*}
All vertical and the upper two horizontal sequences in this diagram are exact. By the Nine-Lemma, the lower horizontal sequence is therefore also short exact. This proves part (i). 

For part (ii), note first that the epimorphism $\Pic(X)\rightarrow\Pic(\widetilde{X})$ preserves degrees and this property descends to $\Pic^{log}(X)\rightarrow\Pic(\calC)$ by Corollary \ref{corollary_degree} above. In order to prove the second part we are going to show that $r(\calL)\geq r$ if and only if $r_\calC(D)\geq r$ for all $r\geq -1$. If $r=-1$ this statement is trivial; so we may assume $r\geq 0$. 

Suppose that $r(\calL)\geq r$. Let $D'$ be a divisor on $\calC$ of degree $r$. Choose a lift $\calL'$ of $D'$, i.e. a line bundle $\calL'$ in $\Pic^r(X)$ on $X$ such that $[\calL']$ maps to $[D]$ in $\Pic^r(\calC)$. Since $\vert \calL\otimes(\calL')^{-1}\vert\neq \emptyset$, there is a combinatorially effective line bundle $\calL''$ on $X$ that is equivalent to $\calL\otimes(\calL')^{-1}$. We may therefore write $\pi^\ast\calL''=\calO_{\widetilde{X}}(D'')$ for a divisor $D''\in \Div(\widetilde{X})=\Div(\calC)$ that is effective and equivalent to $D-D'$. This shows that $\vert D-D'\vert\neq\emptyset$. Hence $r_\calC(D)\geq r$. 

Conversely, assume that $r_\calC(D)\geq r$. Let $\calL'$ be a line bundle on $X$ of degree $r$. Choose a divisor $D'\in\Div(\calC)=\Div(\widetilde{X})$ of degree $r$ such that $\calL'=\calO_{\widetilde{X}}(D')$ and the support of $D'$ does not contain the nodes of $X$. Since $\vert D-D'\vert \neq\emptyset$, there is an effective divisor $D''$ on $\calC$ that is equivalent to $D-D'$. We may assume that the support of $D''$ does not contain any of the points $x_v^f$. With a careful choice of the ``gluing data'' in $(k^\ast)^{b_1(G)}$, the line bundle $\calO_{\widetilde{X}}(D'')$ descends to a combinatorially effective line bundle $\calL''$ on $X$ that is equivalent to $\calL\otimes (\calL')^{-1}$ and therefore $\vert \calL\otimes(\calL')^{-1}\vert \neq \emptyset$. This shows that $r(\calL)\geq r$. 
\end{proof}

\begin{corollary}\label{corollary_combinatorialrank=combinatorialrank}
Suppose that the components $X_v$ of $X$ are all smooth (i.e. they have no self-intersection). Then the combinatorial rank $r(\calL)$ of a line bundle $\calL$ on $X$ is equal to the combinatorial rank of $\calL$ defined in \cite[Section 2.2]{AminiBaker_metrizedcurvecomplexes}
\end{corollary}

\begin{proof}
This is an immediate consequence of \cite[Proposition 2.1]{AminiBaker_metrizedcurvecomplexes} and Proposition \ref{prop_Pic=Pic} above.
\end{proof}

\begin{proof}[Proof of Theorem \ref{thm_logRR}]
We begin by recalling the following two well-known notions:
\begin{itemize}
\item Let $\calC$ be a metrized curve complex (with edge lengths one). Denote by $A_v$ the divisor $\sum_{x_v^f\in\calC_v} x_v^f$ on $\calC_v$ and by $K_v$ a canonical divisor on $\calC_v$. A \emph{canonical divisor} on $\calC$ is given as the sum
\begin{equation*}
K_\calC=\sum_v (A_v + K_v) 
\end{equation*}
(see \cite[Section 3]{AminiBaker_metrizedcurvecomplexes}).
\item Let $X$ be a logarithmically smooth curve over $S=(\Spec k,k^\ast\oplus \N)$ that is logarithmically semistable and vertical. Denote by $\omega_X^{log}$ the relative logarithmic cotangent bundle on $X$ over $S$, as defined in \cite{Kato_logstr} and \cite{Kato_logdef}. It is well-known that in our situation $\omega_X^{log}$ is the dualizing sheaf on $\underline{X}$. 
\end{itemize}

Let $X$ be a logarithmically smooth curve over $S=(\Spec k,k^\ast\oplus \N)$ that is logarithmically semistable and vertical, and denote by $\calC=\calC(X)$ the corresponding metrized curve complex. The following Lemma \ref{lemma_canonical} is well-known.

\begin{lemma}\label{lemma_canonical}
Under the quotient map from Proposition \ref{prop_Pic=Pic}, the image of $[\omega_X^{log}]$ in $\Pic(\calC)$ is the class of the canonical divisor $K_\calC$ on $\calC$. 
\end{lemma}
  
\begin{proof}
Our reasoning proceeds in analogy with \cite[Remark 3.1]{AminiBaker_metrizedcurvecomplexes}. Let $x_e$ be a node of $X$ and denote by $x_{f_i}$ (for $i=1,2$) its pre-images in $\widetilde{X}$. A section of  $\omega^{log}_X$ in a neighborhood of $x_e$ then corresponds to a pair of rational functions $s_1$ and $s_2$ on $\tilde{X}$ with simple poles at $x_{f_i}$ such that $\res (s_1,x_{f_1}) + \res (s_2, x_{f_2})=0$. Therefore we have
\begin{equation*}
\pi^\ast\omega^{log}_X\vert_{\widetilde{X}_v}=\calO_{\tilde{X}_v} (A_v+K_v)
\end{equation*}
for divisors $A_v$ and $K_v$ as above, and this proves our claim. 
\end{proof}

Proposition \ref{prop_Pic=Pic} and Lemma \ref{lemma_canonical} show that Theorem \ref{thm_logRR} is an immediate consequence of Amini and Baker's Riemann-Roch formula for metrized curve complexes \cite[Theorem 3.2]{AminiBaker_metrizedcurvecomplexes}, whose proof follows the strategy originally introduced in \cite{BakerNorine_RiemannRoch} and modified in \cite{MikhalkinZharkov_tropJac}.
\end{proof}

%%%%%%%%%%%%%%%%%%%%%%%%%%%%%%%%%%%%%%%%%%%%%%%%%%%%%%

%%%%%%%%%%%%%%%%%%%%%%%%%%%%%%%%%%%%%%%%%%%%%%%%%%%%%%

%%%%%%%%%%%%%%%%%%%%%%%%%%%%%%%%%%%%%%%%%%%%%%%%%%%%%%

\bibliographystyle{amsalpha}
\bibliography{biblio}{}

\newcommand{\etalchar}[1]{$^{#1}$}
\def\soft#1{\leavevmode\setbox0=\hbox{h}\dimen7=\ht0\advance \dimen7
  by-1ex\relax\if t#1\relax\rlap{\raise.6\dimen7
  \hbox{\kern.3ex\char'47}}#1\relax\else\if T#1\relax
  \rlap{\raise.5\dimen7\hbox{\kern1.3ex\char'47}}#1\relax \else\if
  d#1\relax\rlap{\raise.5\dimen7\hbox{\kern.9ex \char'47}}#1\relax\else\if
  D#1\relax\rlap{\raise.5\dimen7 \hbox{\kern1.4ex\char'47}}#1\relax\else\if
  l#1\relax \rlap{\raise.5\dimen7\hbox{\kern.4ex\char'47}}#1\relax \else\if
  L#1\relax\rlap{\raise.5\dimen7\hbox{\kern.7ex
  \char'47}}#1\relax\else\message{accent \string\soft \space #1 not
  defined!}#1\relax\fi\fi\fi\fi\fi\fi}
\providecommand{\bysame}{\leavevmode\hbox to3em{\hrulefill}\thinspace}
\providecommand{\MR}{\relax\ifhmode\unskip\space\fi MR }
\providecommand{\MRhref}[2]{%
  \href{http://www.ams.org/mathscinet-getitem?mr=#1}{#2}
}
\providecommand{\href}[2]{#2}
\begin{thebibliography}{CCUW17}

\bibitem[AB15]{AminiBaker_metrizedcurvecomplexes}
Omid Amini and Matthew Baker, \emph{Linear series on metrized complexes of
  algebraic curves}, Math. Ann. \textbf{362} (2015), no.~1-2, 55--106.
  \MR{3343870}

\bibitem[AC13]{AminiCaporaso_RiemannRoch}
Omid Amini and Lucia Caporaso, \emph{Riemann-{R}och theory for weighted graphs
  and tropical curves}, Adv. Math. \textbf{240} (2013), 1--23. \MR{3046301}

\bibitem[ACG{\etalchar{+}}13]{Abramovichetal_log&moduli}
Dan Abramovich, Qile Chen, Danny Gillam, Yuhao Huang, Martin Olsson, Matthew
  Satriano, and Shenghao Sun, \emph{Logarithmic geometry and moduli}, Handbook
  of moduli. {V}ol. {I}, Adv. Lect. Math. (ALM), vol.~24, Int. Press,
  Somerville, MA, 2013, pp.~1--61. \MR{3184161}

\bibitem[AGV71]{SGA4}
M.~Artin, A.~Grothendieck, and J.-L. Verdier, \emph{Th{\'{e}}orie des topos et
  cohomologie etale des sch{\'{e}}mas}, Lecture Notes in Mathematics
  \textbf{305}, Springer-Verlag, Berlin, 1971.

\bibitem[Bak08]{Baker_specialization}
Matthew Baker, \emph{Specialization of linear systems from curves to graphs},
  Algebra Number Theory \textbf{2} (2008), no.~6, 613--653, With an appendix by
  Brian Conrad. \MR{2448666}

\bibitem[BJ16]{BakerJensen}
Matthew Baker and David Jensen, \emph{Degeneration of linear series from the
  tropical point of view and applications}, Non-archimedean and tropical
  geometry (Matthew Baker and Sam Payne, eds.), Simons Symposia, vol.~1,
  Springer, 2016, pp.~365--433.

\bibitem[BN07]{BakerNorine_RiemannRoch}
Matthew Baker and Serguei Norine, \emph{Riemann-{R}och and {A}bel-{J}acobi
  theory on a finite graph}, Adv. Math. \textbf{215} (2007), no.~2, 766--788.
  \MR{2355607}

\bibitem[Cap94]{Cap94}
Lucia Caporaso, \emph{A compactification of the universal picard variety over
  the moduli space of stable curves}, J. Amer. Math. Soc. \textbf{7} (1994),
  no.~3, 589--660.

\bibitem[Cap08]{CaporasoNeron}
\bysame, \emph{N{\'e}ron models and compactified picard schemes over the moduli
  stack of stable curves}, Amer. J. Math. (2008), 1--47.

\bibitem[CCUW17]{CavalieriChanUlirschWise_tropstack}
Renzo Cavalieri, Melody Chan, Martin Ulirsch, and Jonathan Wise, \emph{A moduli
  stack of tropical curves}, arXiv:1704.03806 [math] (2017).

\bibitem[EH86]{EH86a}
David Eisenbud and Joe Harris, \emph{Limit linear series: basic theory},
  Invent. Math. \textbf{85} (1986), no.~2, 337--371.

\bibitem[Fos15]{F15}
T.~Foster, \emph{Introduction to adic tropicalization}, arXiv:1506.00726
  (2015).

\bibitem[FP17]{FP15}
T.~Foster and S.~Payne, \emph{{Limits of tropicalizations II: structure
  sheaves, adic spaces and cofinality of Gubler models}}, In preparation
  (2017).

\bibitem[FR15]{FR15}
Tyler Foster and Dhruv Ranganathan, \emph{Hahn analytification and connectivity
  of higher rank tropical varieties}, Manuscripta Mathematica (2015), 1--22.

\bibitem[HKN13]{HladkyKralNorine_ranksofdivisors}
Jan Hladk{\'y}, Daniel Kr{\'a}{\soft{l}}, and Serguei Norine, \emph{Rank of
  divisors on tropical curves}, J. Combin. Theory Ser. A \textbf{120} (2013),
  no.~7, 1521--1538. \MR{3092681}

\bibitem[Ill94]{Illusie_logspaces}
Luc Illusie, \emph{Logarithmic spaces (according to {K}. {K}ato)}, Barsotti
  {S}ymposium in {A}lgebraic {G}eometry ({A}bano {T}erme, 1991), Perspect.
  Math., vol.~15, Academic Press, San Diego, CA, 1994, pp.~183--203.
  \MR{1307397}

\bibitem[Kaj93]{Kajiwara_logJacobian}
Takeshi Kajiwara, \emph{Logarithmic compactifications of the generalized
  {J}acobian variety}, J. Fac. Sci. Univ. Tokyo Sect. IA Math. \textbf{40}
  (1993), no.~2, 473--502. \MR{1255052}

\bibitem[Kat89]{Kato_logstr}
Kazuya Kato, \emph{Logarithmic structures of {F}ontaine-{I}llusie}, Algebraic
  analysis, geometry, and number theory ({B}altimore, {MD}, 1988), Johns
  Hopkins Univ. Press, Baltimore, MD, 1989, pp.~191--224. \MR{1463703}

\bibitem[Kat96]{Kato_logdef}
Fumiharu Kato, \emph{Log smooth deformation theory}, Tohoku Math. J. (2)
  \textbf{48} (1996), no.~3, 317--354. \MR{1404507}

\bibitem[Kat00]{Kato_logsmoothcurves}
\bysame, \emph{Log smooth deformation and moduli of log smooth curves},
  Internat. J. Math. \textbf{11} (2000), no.~2, 215--232. \MR{1754621}

\bibitem[Len17]{Len16}
Yoav Len, \emph{Hyperelliptic graphs and metrized complexes}, Forum Math. Sigma
  \textbf{5} (2017), e20, 15. \MR{3692877}

\bibitem[Luo11]{Luo_rankdeterminingsets}
Ye~Luo, \emph{Rank-determining sets of metric graphs}, J. Combin. Theory Ser. A
  \textbf{118} (2011), no.~6, 1775--1793. \MR{2793609}

\bibitem[MZ08]{MikhalkinZharkov_tropJac}
Grigory Mikhalkin and Ilia Zharkov, \emph{Tropical curves, their {J}acobians
  and theta functions}, Curves and abelian varieties, Contemp. Math., vol. 465,
  Amer. Math. Soc., Providence, RI, 2008, pp.~203--230. \MR{2457739}

\bibitem[Ols04]{Olsson_K3}
Martin~C. Olsson, \emph{Semistable degenerations and period spaces for
  polarized {$K3$} surfaces}, Duke Math. J. \textbf{125} (2004), no.~1,
  121--203. \MR{2097359}

\bibitem[Ols07]{Olsson_logtwistedcurves}
\bysame, \emph{({L}og) twisted curves}, Compos. Math. \textbf{143} (2007),
  no.~2, 476--494. \MR{2309994}

\bibitem[Par12]{Par12b}
Brett Parker, \emph{Log geometry and exploded manifolds}, Abhandlungen aus dem
  Mathematischen Seminar der Universit{\"a}t Hamburg, vol.~82, Springer, 2012,
  pp.~43--81.

\bibitem[Ray70]{Ray70}
Michel Raynaud, \emph{Sp{\'e}cialisation du foncteur de picard}, Publications
  Math{\'e}matiques de l'IH{\'E}S \textbf{38} (1970), 27--76.

\bibitem[Uli17]{U13}
Martin Ulirsch, \emph{Functorial tropicalization of logarithmic schemes: the
  case of constant coefficients}, Proc. Lond. Math. Soc. (3) \textbf{114}
  (2017), no.~6, 1081--1113. \MR{3661346}

\end{thebibliography}

\end{document}